\documentclass[10pt]{amsart}
\usepackage{amssymb}
\usepackage{amsthm,amsmath}

\title[Reliability of systems with dependent components]{Reliability of systems with dependent components based on lattice polynomial description}

\author{Alexander Dukhovny}
\address{Mathematics Department, San Francisco State University \\
San Francisco, CA 94132, USA} \email{dukhovny[at]math.sfsu.edu}

\author{Jean-Luc Marichal}
\address{Mathematics Research Unit, FSTC, University of Luxembourg \\
6, rue Coudenhove-Kalergi, L-1359 Luxembourg, Luxembourg} \email{jean-luc.marichal[at]uni.lu }

\date{March 25, 2011}

\begin{document}

\newenvironment{disarray}%
 {\everymath{\displaystyle\everymath{}}\array}%
 {\endarray}

\theoremstyle{plain}
\newtheorem{theorem}{Theorem}
\newtheorem{lemma}[theorem]{Lemma}
\newtheorem{proposition}[theorem]{Proposition}
\newtheorem{corollary}[theorem]{Corollary}
\newtheorem{fact}[theorem]{Fact}
\newtheorem*{main}{Main Theorem}

\theoremstyle{definition}
\newtheorem{definition}[theorem]{Definition}
\newtheorem{example}[theorem]{Example}

\theoremstyle{remark}
\newtheorem*{conjecture}{\indent Conjecture}
\newtheorem{remark}{Remark}
\newtheorem{claim}{Claim}

\newcommand{\N}{\mathbb{N}}                     
\newcommand{\Z}{\mathbb{Z}}                     
\newcommand{\R}{\mathbb{R}}                     
\newcommand{\Vspace}{\vspace{2ex}}                  
\newcommand{\Ind}{\mathrm{Ind}}

\begin{abstract}
Reliability of a system is considered where the components' random lifetimes may be dependent. The structure of the system is described by an
associated ``lattice polynomial'' function. Based on that descriptor, general framework formulas are developed and used to obtain direct results
for the cases where a) the lifetimes are ``Bayes-dependent'', that is, their interdependence is due to external factors (in particular, where
the factor is the ``preliminary phase'' duration) and b) where the lifetimes' dependence is implied by upper or lower bounds on lifetimes of
components in some subsets of the system. (The bounds may be imposed externally based, say, on the connections environment.) Several special
cases are investigated in detail.
\end{abstract}

\keywords{Reliability, semicoherent system, lattice polynomial function.}

\subjclass[2010]{62N05, 90B25 (Primary) 28B15, 94C10 (Secondary)}

\maketitle

\section{Introduction}

A semicoherent system consisting of nonrepairable components with random lifetimes can be associated with a Boolean function called the
\emph{structure function} of the system. It expresses the ``on'' indicator of the system through the ``on'' indicators of the components based
on the logic of connections. In turn, the structure function can be extended to a lattice polynomial function called the \emph{life function} of
the system, which expresses the system lifetime in terms of the component lifetimes using minimum and maximum in place of conjunction and
disjunction, respectively.

In this paper we use the results on lattice polynomial (l.p.) functions from a series of our recent articles to obtain formulas for system
reliability where the components' lifetimes are dependent random variables.

In the current literature, including such prominent sources as \cite{BarPro81,BarPro96,KauGroCru77,Ram90}, that case is explored only under the
additional assumption of exchangeable components (Barlow and Proschan \cite{BarPro81,BarPro96}) which presupposes that the marginal
distributions of components' lifetimes are identical. Meanwhile, in many realistic systems that assumption cannot be applied.

Using l.p.\ functions we derive a general formula for system reliability that yields specific results for some important special cases of
dependence among components.

One is the case where the dependence can be ascribed to the simultaneous influence on all components of some random factors (such as, say, the
system's environment physical parameters: pressure, temperature, moisture level, etc). For any given set of values of those factors components'
lifetimes are conditionally independent (but their probability distributions depend on the factors' values). We will refer to this case as
``Bayes-dependence''.

In particular, of special interest is a model where there is a preliminary period (``pre-phase'') of random duration and all components are
guaranteed to survive it. Once it is over, the components' individual ``after-phases'' are independent and their probability distributions
depend on the ``pre-phase'' duration.

Another pattern of dependence arises when there are collective upper or lower bounds on lifetimes of certain subsets of components imposed by
external conditions such as physical properties of the assembly or its environment. Say, when some components of the system are connected to the
same power source, the lifetime of the source (or the fuse) becomes a collective upper bound for those components' lifetimes. Or, when a part of
the system is backed up by a reserve standby device (too expensive, perhaps, for a regular duty) that instantly picks up duties of any failed
component in that subset, the device's lifetime becomes a lower bound for the components' lifetimes in that subset. A ``senior'' air traffic
controller could serve as an example here.

To our knowledge, models with such patterns of components' interdependence have not been analyzed in the literature. Using l.p.\ description of
a system in this paper presents a natural framework for such analysis because each component's service duration can be easily and naturally
represented using an l.p.\ function involving the component's own lifetime and bounding random variables.

Moreover, the case where the collective bounds are constant can be analyzed through so-called weighted lattice polynomial (w.l.p.) functions.

This paper is organized as follows. In Section 2 we introduce the l.p.\ function of a semicoherent system and use it to present a number of
reliability formulas for a system with generally dependent components' lifetimes.

In Section 3 we describe the case of ``Bayes-dependence'' of components' lifetimes and present exact formulas for the system reliability. In
turn, those formulas make it possible to provide exact formulation of reliability parameters such as the mean time-to-failure of the system.

In Section 4 we analyze systems with ``pre-phase'' dependent components. In particular, we present closed-form results for the case where the
``after-phase'' durations of components are assumed to have exponential distributions with failure rates depending on the ``pre-phase''
duration. We consider this model as giving a system analyst an alternative to using Weibull distributions, more complex both analytically and
statistically.

In Section 5 we consider systems with lower and/or upper bounds on service duration of (subsets) of components. There, we show how the system's
reliability can be computed as reliability of another, ``augmented'' system with added components and connections but without subset bounds.
Based on that representation, when the original components intrinsic lifetimes are independent, the augmented system's reliability is given by
formulas of Sections 2 and 3.

For any numbers $\alpha,\beta\in \overline{\R}=[-\infty,\infty]$ and any subset $A\subseteq [n]=\{1,\ldots,n\}$, let
$\mathbf{e}^{\alpha,\beta}_A$ denote the characteristic vector of $A$ in $\{\alpha,\beta\}^n$, that is, the $n$-tuple whose $i$th coordinate is
$\beta$, if $i\in A$, and $\alpha$, otherwise.

\section{Lattice polynomial function and reliability of a system}
\label{sec:LatPolF-SysRel}

Consider a system consisting of $n$ components that are interconnected. The \emph{state} of a component $i\in [n]$ can be represented by a
Boolean variable $x_i$ defined as
$$
x_i =
\begin{cases}
1, & \mbox{if component $i$ is functioning,}\\
0, & \mbox{if component $i$ is in a failed state.}
\end{cases}
$$
For convenience, we also introduce the state vector $\mathbf{x}=(x_1,\ldots,x_n)$.

As is common in the literature, the \emph{state of the system} is described from the component states through a Boolean function
$\phi\colon\{0,1\}^n\to\{0,1\}$, called the \emph{structure function} of the system and defined as
$$
\phi(\mathbf{x}) =
\begin{cases}
1, & \mbox{if the system is functioning,}\\
0, & \mbox{if the system is in a failed state.}
\end{cases}
$$

We shall assume throughout that the structure function $\phi$ is nondecreasing (the system is then said to be \emph{semicoherent}) and
nonconstant, this latter condition ensuring that $\phi(\mathbf{0})=0$ and $\phi(\mathbf{1})=1$, where $\mathbf{0}=(0,\ldots,0)$ and
$\mathbf{1}=(1,\ldots,1)$.

As a Boolean function, the structure function $\phi$ can also be regarded as a set function $v\colon 2^{[n]}\to\{0,1\}$. The correspondence is
straightforward: We have $v(A)=\phi(\mathbf{e}^{0,1}_A)$ for all $A\subseteq [n]$ and
\begin{equation}\label{eq:PhiForm1}
\phi(\mathbf{x}) = \sum_{A\subseteq [n]}v(A)\prod_{i\in A}x_i\prod_{i\in [n]\setminus A}(1-x_i).
\end{equation}
We shall henceforth make this identification and often write $\phi_v(\mathbf{x})$ instead of $\phi(\mathbf{x})$. Clearly, the structure function
$\phi_v$ is nondecreasing and nonconstant if and only if its underlying set function $v$ is nondecreasing and nonconstant.

Another concept that we shall often use in this paper is the \emph{dual} of the set function $v$, that is, the set function $v^*\colon
2^{[n]}\to\{0,1\}$ defined by
$$
v^*(A)=1-v([n]\setminus A).
$$

For any event $E$, let $\mathrm{Ind}(E)$ represent the \emph{indicator random variable} that gives $1$ if $E$ occurs and $0$ otherwise. For any
$i\in [n]$, we denote by $T_i$ the random \emph{time-to-failure} of component $i$ and we denote by $X_i(t)=\Ind(T_i>t)$ the random \emph{state
at time $t\geqslant 0$} of component $i$. For simplicity, we introduce the random time-to-failure vector $\mathbf{T}=(T_1,\ldots,T_n)$ and the
random state vector $\mathbf{X}(t)=(X_1(t),\ldots,X_n(t))$ at time $t\geqslant 0$. We also denote by $T_S$ the random time-to-failure of the
system and by $X_S(t)=\Ind(T_S>t)$ the random state at time $t\geqslant 0$ of the system.

The structure function $\phi$ clearly induces a functional relationship between the variables $T_1,\ldots,T_n$ and the variable $T_S$. As we
will see in Theorem~\ref{thm:SFvsLP}, $T_S$ is always an l.p.\ function of the variables $T_1,\ldots,T_n$. Just as for the structure function,
this l.p.\ function provides a complete description of the structure of the system.

Let us first recall the concept of l.p.\ function of real variables; see for instance Birkhoff~\cite{Bir67} and Gr\"atzer~\cite{Grae03}. Let
$L\subseteq\overline{\R}$ denote a totally ordered bounded lattice whose lattice operations $\wedge$ and $\vee$ are respectively the minimum and
maximum operations. Denote also by $a$ and $b$ the bottom and top elements of $L$. We assume that $a\neq b$.

\begin{definition}
The class of \emph{lattice polynomial} (l.p.) functions from $L^n$ to $L$ is defined as follows:
\begin{enumerate}
\item[(i)] For any $k\in [n]$, the projection $\mathbf{t}=(t_1,\ldots,t_n)\mapsto t_k$ is an l.p.\ function from $L^n$ to $L$.

\item[(ii)] If $p$ and $q$ are l.p.\ functions from $L^n$ to $L$, then $p\wedge q$ and $p\vee q$ are l.p.\ functions from $L^n$ to $L$.

\item[(iii)] Every l.p.\ function from $L^n$ to $L$ is constructed by finitely many applications of the rules (i) and (ii).
\end{enumerate}
\end{definition}

Clearly, any l.p.\ function $p\colon L^n\to L$ is nondecreasing and nonconstant. Furthermore, it was proved (see for instance \cite{Bir67}) that such a function can be expressed in \emph{disjunctive} and \emph{conjunctive} normal forms, that is, there always
exist nonconstant set functions $w^d\colon 2^{[n]}\to\{a,b\}$ and $w^c\colon 2^{[n]}\to\{a,b\}$, with $w^d(\varnothing)=a$ and
$w^c(\varnothing)=b$, such that
\begin{equation}\label{eq:LPDisjConj}
p(\mathbf{t})=\bigvee_{\textstyle{A\subseteq [n]\atop w^d(A)=b}}\bigwedge_{i\in A} t_i = \bigwedge_{\textstyle{A\subseteq [n]\atop
w^c(A)=a}}\bigvee_{i\in A} t_i.
\end{equation}

Clearly, the set functions $w^d$ and $w^c$ that disjunctively and conjunctively define the l.p.\ function $p$ in (\ref{eq:LPDisjConj}) are not
unique. However, it can be shown \cite{Marc} that, from among all the possible set functions that disjunctively define $p$, only one is
nondecreasing. Similarly, from among all the possible set functions that conjunctively define $p$, only one is nonincreasing. These special set
functions are given by
$$
w^d(A)=p(\mathbf{e}^{a,b}_A)\qquad\mbox{and}\qquad w^c(A)=p(\mathbf{e}^{a,b}_{[n]\setminus A}).
$$
The l.p.\ function disjunctively defined by a given nondecreasing set function $w\colon 2^{[n]}\to\{a,b\}$ will henceforth be denoted $p_w$. We
then have
$$
p_w(\mathbf{t})=\bigvee_{\textstyle{A\subseteq [n]\atop w(A)=b}}\bigwedge_{i\in A} t_i = \bigwedge_{\textstyle{A\subseteq [n]\atop
w^*(A)=b}}\bigvee_{i\in A} t_i,
$$
where $w^*$ is the \emph{dual} of $w$, defined as $w^*=\gamma\circ (\gamma^{-1}\circ w)^*$, the function $\gamma\colon\{0,1\}\to\{a,b\}$ being a simple
transformation defined by $\gamma(0)=a$ and $\gamma(1)=b$.

\begin{remark}\label{rem:7s6ad}
From any nonconstant and nondecreasing set function $w\colon 2^{[n]}\to\{a,b\}$, define the set function $u_w\colon 2^{[n]}\to\{a,b\}$ as
$$
u_w(A)=
\begin{cases}
b, & \mbox{if $w(A)=b$ and $w(B)=a$ for all $B\varsubsetneq A$,}\\
a, & \mbox{otherwise}.
\end{cases}
$$
The disjunctive and conjunctive representations of the l.p.\ function $p_w$ having a minimal number of terms are given by
$$
p_w(\mathbf{t})=\bigvee_{\textstyle{A\subseteq [n]\atop u_w(A)=b}}\bigwedge_{i\in A} t_i = \bigwedge_{\textstyle{A\subseteq [n]\atop
u_{w^*}(A)=b}}\bigvee_{i\in A} t_i
$$
(see Proposition~8 in \cite{Marc}) and are exactly those  minimal paths and cuts representations of $p_w$ (as given, say in Theorem 3.5 of
\cite{BarPro81}), namely
$$
p_w(\mathbf{t})=\bigvee_{j=1}^r\bigwedge_{i\in P_j}t_i = \bigwedge_{j=1}^s\bigvee_{i\in K_j}t_i\, ,
$$
where $P_1,\ldots,P_r$ are the minimal path sets and $K_1,\ldots,K_s$ are the minimal cut sets.
\end{remark}

The following theorem points out the one-to-one correspondence between the structure function and the l.p.\ function that expresses $T_S$
directly in terms of the variable $T_1,\ldots,T_n$. It is the latter fact that makes, in our opinion, the l.p.\ function a preferable system
descriptor.

As lifetimes are $[0,\infty]$-valued, we shall henceforth assume without loss of generality that $L=[0,\infty]$, that is, $a=0$ and $b=\infty$.
Recall that the coproduct of $n$ Boolean variables $x_1,\ldots,x_n$ is defined by $\amalg_i x_i=1-\Pi_i(1-x_i)$.

\begin{theorem}\label{thm:SFvsLP}
Consider an $n$-component system whose structure function $\phi$ is nondecreasing and nonconstant. Then we have
\begin{equation}\label{eq:TpT}
T_S=p_w(T_1,\ldots,T_n),
\end{equation}
where $w=\gamma\circ v$. Conversely, any system fulfilling (\ref{eq:TpT}) for some l.p.\ function $p_w\colon L^n\to L$ has the nondecreasing and
nonconstant structure function $\phi_v$, where $v=\gamma^{-1}\circ w$.
\end{theorem}

\begin{proof}
The proof mainly lies on the distributive property of the indicator function $\Ind(\cdot)$ with respect to disjunction and conjunction: for any
events $E$ and $E'$,
\begin{eqnarray*}
\Ind(E\vee E') &=& \Ind(E)\vee\Ind(E'),\\
\Ind(E\wedge E') &=& \Ind(E)\wedge\Ind(E').
\end{eqnarray*}
Thus, for any $t\geqslant 0$ we have
\begin{eqnarray*}
\Ind(p_w(\mathbf{T})>t) &=& \Ind\Big(\bigvee_{\textstyle{A\subseteq [n]\atop v(A)=1}}\bigwedge_{i\in A} T_i>t\Big)
~=~ \bigvee_{\textstyle{A\subseteq [n]\atop v(A)=1}}\bigwedge_{i\in A} \Ind(T_i>t)\\
&=& \coprod_{\textstyle{A\subseteq [n]\atop v(A)=1}}\prod_{i\in A} X_i(t)\\
&=& \phi_v(\mathbf{X}(t)).
\end{eqnarray*}
Hence, we have $T_S=p_w(\mathbf{T})$ if and only if $X_S(t)=\phi_v(\mathbf{X}(t))$ for all $t\geqslant 0$, which completes the proof.
\end{proof}


The \emph{reliability function} of component $i$ is defined, for any $t\geqslant 0$, by
$$
R_i(t)=\Pr(T_i>t)=\Pr(X_i(t)=1)=\mathrm{E}[X_i(t)],
$$
that is, the probability that component $i$ does not fail in the time interval $[0,t]$. Similarly, for any $t\geqslant 0$, the system
reliability function is
$$
R_S(t)=\Pr(T_S>t)=\Pr(X_S(t)=1)=\mathrm{E}[X_S(t)],
$$
that is, the probability that the system does not fail in the time interval $[0,t]$.

Based on representation (\ref{eq:PhiForm1}) and its dual form, we present general formulas for the system reliability function in case of
generally dependent variables $T_1,\ldots,T_n$ (first presented in \cite{Duk07}).

\begin{theorem}\label{thm:RSv}
We have
\begin{eqnarray}
R_S(t) &=& \sum_{A\subseteq [n]} v(A)\,\Pr(\mathbf{X}(t)=\mathbf{e}^{0,1}_A),\label{eq:RSv}\\
R_S(t) &=& 1-\sum_{A\subseteq [n]} v^*(A)\,\Pr(\mathbf{X}(t)=\mathbf{e}^{0,1}_{[n]\setminus A})\label{eq:RSvs}.
\end{eqnarray}
\end{theorem}

\begin{proof}
By (\ref{eq:PhiForm1}), we have
\begin{eqnarray}
R_S(t) ~=~ \mathrm{E}[\phi_v(\mathbf{X}(t))] &=& \sum_{A\subseteq [n]}v(A)\,\mathrm{E}\Big[\prod_{i\in A}X_i(t)\prod_{i\in
[n]\setminus A}(1-X_i(t))\Big]\label{eq:RSv1}\\
&=& \sum_{A\subseteq [n]} v(A)\,\Pr(\mathbf{X}(t)=\mathbf{e}^{0,1}_A),\nonumber
\end{eqnarray}
which proves (\ref{eq:RSv}). Formula (\ref{eq:RSvs}) can be proved similarly by using the dual form of $\phi_v$ (i.e., the second expression in
Table~\ref{tab:StrFun}).
\end{proof}

More reliability function formulas can be obtained based on other structure function representations. Any Boolean function has a unique
expression as a multilinear function in $n$ variables,
\begin{equation}\label{eq:PhiForm2}
\phi_v(\mathbf{x}) = \sum_{A\subseteq [n]}m_v(A)\prod_{i\in A}x_i
\end{equation}
(see for instance \cite{HamRud68}), where the set function $m_v\colon 2^{[n]}\to\Z$ is the \emph{M\"obius transform} of $v$,
defined by
$$
m_v(A)=\sum_{B\subseteq A}(-1)^{|A|-|B|}\, v(B).
$$

By using the dual set function $v^*$ we can easily derive further useful forms of the structure function. Table~\ref{tab:StrFun} summarizes the
best known forms of the structure function.

\begin{table}[tbp]
$$
\begin{array}{|c|c|}
\hline \mbox{Name} & \phi_v(\mathbf{x}) \\
\hline & \\
\mbox{Primal form} & \sum\limits_{A\subseteq [n]}v(A)\prod\limits_{i\in A}x_i\prod\limits_{i\in [n]\setminus A}(1-x_i)\\ &\\
\mbox{Dual form} & 1-\sum\limits_{A\subseteq [n]}v^*(A)\prod\limits_{i\in [n]\setminus A}x_i\prod\limits_{i\in A}(1-x_i)\\ &\\
\mbox{Primal M\"obius form} & \sum\limits_{A\subseteq [n]}m_v(A)\prod\limits_{i\in A}x_i\\ &\\
\mbox{Dual M\"obius form} & \sum\limits_{A\subseteq [n]}m_{v^*}(A)\coprod\limits_{i\in A}x_i\\ &\\
\mbox{Disjunctive normal form} & \coprod\limits_{A\subseteq [n]} v(A)\prod\limits_{i\in A} x_i\\ &\\
\mbox{Conjunctive normal form} & \prod\limits_{A\subseteq [n]_{\mathstrut}} v^*(A)\coprod\limits_{i\in A} x_i\\
\hline
\end{array}
$$
\caption{Various forms of the structure function} \label{tab:StrFun}
\end{table}

\begin{remark}
Since $\phi_v$ is a Boolean function, we can always replace in its expression each product $\Pi$ and coproduct $\amalg$ with the minimum
$\wedge$ and the maximum $\vee$, respectively. Thus, Theorem~\ref{thm:SFvsLP} essentially states that $\phi_v$ is also an l.p.\ function that
has just the same max-min form as $p_w$ but applied to binary arguments. More precisely, $\phi_v$ is \emph{similar} to $p_w$ in the sense that
$\gamma\circ\phi_v=p_w\circ(\gamma,\ldots,\gamma)$.
\end{remark}

Consider the \emph{joint distribution function} and the \emph{joint survival function}, defined respectively as
$$
F(\mathbf{t}) = \Pr(T_i\leqslant t_i\;\forall i\in [n])\qquad\mbox{and}\qquad R(\mathbf{t}) = \Pr(T_i>t_i\;\forall i\in [n]).
$$
By using the same argument as in the proof of Theorem~\ref{thm:RSv}, we obtain two further equivalent expressions of $R_S(t)$.

\begin{theorem}\label{thm:RSmvs}
We have
\begin{eqnarray*}
R_S(t) &=& \sum_{A\subseteq [n]} m_v(A) \, R(\mathbf{e}^{0,t}_A)\label{eq:RSmv}\\
R_S(t) &=& 1-\sum_{A\subseteq [n]} m_{v^*}(A) \, F(\mathbf{e}^{t,\infty}_{[n]\setminus A}).\label{eq:RSmvs}
\end{eqnarray*}
\end{theorem}

\begin{proof}
By (\ref{eq:PhiForm2}), we have
\begin{eqnarray*}
R_S(t) = \mathrm{E}[\phi_v(\mathbf{X}(t))] &=& \sum_{A\subseteq [n]} m_v(A) \,\mathrm{E}\Big[\prod_{i\in A}X_i(t)\Big]\\
&=& \sum_{A\subseteq [n]} m_v(A) \,\Pr(T_i>t\;\forall i\in A)\\
&=& \sum_{A\subseteq [n]} m_v(A) \,R(\mathbf{e}^{0,t}_A).
\end{eqnarray*}
Similarly, using the dual M\"obius form of $\phi_v$ (i.e., the fourth expression in Table~\ref{tab:StrFun}), we have
\begin{eqnarray*}
R_S(t) = \mathrm{E}[\phi_v(\mathbf{X}(t))] &=& \sum_{A\subseteq [n]} m_{v^*}(A) \,\mathrm{E}\Big[\coprod_{i\in A}X_i(t)\Big]\\
&=& 1-\sum_{A\subseteq [n]} m_{v^*}(A) \,\mathrm{E}\Big[\prod_{i\in A}(1-X_i(t))\Big]\\
&=& 1-\sum_{A\subseteq [n]} m_{v^*}(A) \,\Pr(T_i\leqslant t\;\forall i\in A)\\
&=& 1-\sum_{A\subseteq [n]} m_{v^*}(A) \,F(\mathbf{e}^{t,\infty}_{[n]\setminus A}),
\end{eqnarray*}
where we have used the fact that $\sum_{A\subseteq [n]} m_{v^*}(A) = \phi_{v^*}(\mathbf{1})=1$.
\end{proof}

\begin{remark}
Based on the minimal path/cut sets representations (see Remark~\ref{rem:7s6ad}), one can find in \cite{AgrBar84} and further in, say, \cite{BloLiSav03} and \cite{NavRuiSan07} linear representations of $R_S(t)$ in terms of reliability functions ``series'' and/or ``parallel'' subsystems. In this paper we chose to use the representations of Theorem~\ref{thm:RSv}, among other reasons, in order to use the M\"obius transform and obtain two further equivalent expressions of $R_S(t)$.
\end{remark}

The \emph{mean time-to-failure of component $i$} is defined as $\mathrm{MTTF}_i=\mathrm{E}[T_i]$ and similarly the \emph{mean time-to-failure of
the system} is defined as $\mathrm{MTTF}_S=\mathrm{E}[T_S]$. These expected values can be calculated by the following formulas (see for instance \cite{RauHoy04})
$$
\mathrm{MTTF}_i=\int_0^{\infty} R_i(t)\, dt \qquad\mbox{and}\qquad \mathrm{MTTF}_S=\int_0^{\infty} R_S(t)\, dt.
$$

It is noteworthy that Theorem~\ref{thm:RSmvs} immediately provides concise expressions for the mean time-to-failure of the system, namely
\begin{eqnarray*}
\mathrm{MTTF}_S &=& \sum_{A\subseteq [n]} m_v(A) \, \int_0^{\infty} R(\mathbf{e}^{0,t}_A)\, dt,\\
\mathrm{MTTF}_S &=& \sum_{A\subseteq [n]} m_{v^*}(A) \, \int_0^{\infty} \big(1-F(\mathbf{e}^{t,\infty}_{[n]\setminus A})\big)\, dt.
\end{eqnarray*}

Theorem~\ref{thm:RSmvs} may suggest that the complete knowledge of the joint survival (or joint distribution) function is needed for the
calculation of the system reliability function.  Actually, as Theorem~\ref{thm:RSv} shows, all the needed information is encoded in the
distribution of the indicator vector $\mathbf{X}(t)$. In turn, the distribution of $\mathbf{X}(t)$ can be easily expressed (see
\cite{Duk07,DukMar08}) in terms of the joint distribution function of $\mathbf{X}(t)$, which is defined by
$$
\Pr(\mathbf{X}(t)=\mathbf{e}_A^{0,1}) = \sum_{B\subseteq A} (-1)^{|A|-|B|}\, F(\mathbf{e}_B^{t,\infty}).
$$
That is, the joint distribution function needs to be known only where its arguments are equal to $t$ or $\infty$.

In the case when $T_1,\ldots,T_n$ are independent, which implies that the indicator variables $X_1(t),\ldots,X_n(t)$ are independent for all
$t\geqslant 0$, from (\ref{eq:RSv1}) we immediately obtain (see for instance \cite{RauHoy04}):
\begin{equation}\label{eq:RsRiInd}
R_S(t) = \sum_{A\subseteq [n]} v(A) \prod_{i\in A}R_i(t)\prod_{i\in [n]\setminus A}(1-R_i(t)).
\end{equation}

By extending formally the structure function $\phi_v$ to $[0,1]^n$ by linear interpolation, we define the \emph{multilinear extension}\/ of
$\phi_v$ (a concept introduced in game theory by Owen~\cite{Owe72}), that is, the multilinear function $\overline{\phi}_v\colon [0,1]^n\to
[0,1]$ defined as
\begin{equation}\label{eq:MLE1}
\overline{\phi}_v(\mathbf{x}) = \sum_{A\subseteq [n]}v(A)\prod_{i\in A}x_i\prod_{i\in [n]\setminus A}(1-x_i).
\end{equation}

We then observe that each of the alternative expressions of $\phi_v$ introduced in Table~\ref{tab:StrFun} can be formally regarded as a function
from $[0,1]^n$ to $[0,1]$, which then identifies with the multilinear extension of $\phi_v$.

Combining (\ref{eq:RsRiInd}) and (\ref{eq:MLE1}), we retrieve the classical formula (see for instance \cite{ProSet76,RauHoy04})\footnote{A similar expression was obtained for the general case in Section 5 of \cite{NavSpi10}.}
$$
R_S(t) = \overline{\phi}_v(R_1(t),\ldots,R_n(t))
$$
where the function $\overline{\phi}_v$ is called the \emph{reliability polynomial}. Thus, both $R_S(t)$ and $\mathrm{MTTF}_S$ can be expressed
in different forms, according to the expressions of $\overline{\phi}_v$ corresponding to Table~\ref{tab:StrFun}. For instance, using the primal
M\"obius form of $\overline{\phi}_v$, we obtain
\begin{eqnarray*}
R_S(t) &=& \sum_{A\subseteq [n]} m_v(A) \, \prod_{i\in A} R_i(t),\\
\mathrm{MTTF}_S &=& \sum_{A\subseteq [n]} m_v(A) \, \int_0^{\infty} \prod_{i\in A} R_i(t)\, dt.
\end{eqnarray*}

Consider the special case when the reliability of every subset $A\subseteq [n]$ depends only on the number $|A|$ of components in $A$ (which
happens, for example, when the component lifetimes are exchangeable). That is,
$$
R(\mathbf{e}^{0,t}_A)=R(\mathbf{e}^{0,t}_{A'})\quad\mbox{whenever}\quad |A|=|A'|,
$$
and similarly for $F(\mathbf{e}^{t,\infty}_{[n]\setminus A})$.

Defining $R(k,t) = R(\mathbf{e}^{0,t}_A)$, $F(k,t) = F(\mathbf{e}^{t,\infty}_{[n]\setminus A})$, where $k=|A|$, and
$$
\overline{m}_v(k) = \sum_{A\subseteq [n]\, :\, |A|=k} m_v(A),
$$
from Theorem~\ref{thm:RSmvs} we derive the following corollary.\footnote{For the case of exchangeable component lifetimes, this corollary was
actually obtained in \cite{NavRuiSan07} by using the concept of signature.}

\begin{corollary}
If the reliability of every subset depends only on the number of components in the subset, then
\begin{eqnarray*}
R_S(t) &=& \sum_{k=1}^n\overline{m}_v(k)\, R(k,t)\\
R_S(t) &=& 1-\sum_{k=1}^n\overline{m}_{v^*}(k)\, F(k,t).
\end{eqnarray*}
\end{corollary}

\section{Systems with ``Bayes-dependent'' component lifetimes}

Consider a system where components' functioning is influenced by certain (perhaps, random) factors that may be both internal and external to the
system. Say, failure rates of individual components may be influenced by the system's current environment conditions, such as temperature,
pressure, precipitation, etc. Or, the whole assembly is made of subsystems that have ``central'' units and their status directly affects the
other units of the subsystem.

One way to model that situation is to introduce a set of (random) factors $U_1,\ldots,U_m$ whose joint probability density function
$g(u_1,\ldots,u_m)$ defined on a domain $D_g$ is known. Given a fixed set of the factor's values, the system's components' lifetimes are assumed
independent with individual conditional cumulative distribution functions $F_i(t,u_1,\ldots,u_m)$ depending on the factors' values. In that
case, the joint probability distribution function of the components' lifetimes acquires an ``integrated product'' form:
\begin{equation}\label{eq:CdfQuasiProd}
F(\mathbf{t}) = \int_{D_g} g(\mathbf{u})\,\prod_{i\in [n]} F_i(t_i,\mathbf{u})\, d\mathbf{u}.
\end{equation}
We will refer to such interdependence pattern as ``Bayes-dependence''.

Clearly, when values of factors $U_1,\ldots, U_m$ are fixed, since $T_1,\ldots, T_m$ are (conditionally) independent, the indicator variables $X_1(t),\ldots, X_m(t)$ are too, so using (\ref{eq:CdfQuasiProd}) in (\ref{eq:RSv1}) with fixed $\mathbf{u}$ the same way as in the derivation of (\ref{eq:RsRiInd}) we then integrate over all possible values of $\mathbf{u}$ and obtain a generalization of (\ref{eq:RsRiInd}):
\begin{eqnarray*}
R_S(t) &=& \int_{D_g} g(\mathbf{u})\,\Pr(T_S > t \mid \mathbf{u})\, d\mathbf{u}\\
&=& \int_{D_g} g(\mathbf{u})\,\sum_{A\subseteq [n]} v(A) \prod_{i\in A}R_i(t, \mathbf{u})\prod_{i\in [n]\setminus A}\big(1-R_i(t,
\mathbf{u})\big)\, d\mathbf{u},
\end{eqnarray*}
where $R_i(t,\mathbf{u}) = 1- F_i(t,\mathbf{u})$.

In turn, it leads to a generalization of the classical reliability formula
$$
R_S(t)=\int_{D_g} g(\mathbf{u})\, \overline{\phi}_v(R_1(t,\mathbf{u}),\ldots,R_n(t,\mathbf{u}))\, d\mathbf{u}.
$$

Once again, both $R_S(t)$ and $\mathrm{MTTF}_S$ can be expressed in different forms, according to the expressions of $\overline{\phi}_v$
corresponding to Table~\ref{tab:StrFun}. In particular, it is very convenient for calculations to use the primal M\"obius form of
$\overline{\phi}_v$:
\begin{eqnarray}
R_S(t) &=& \sum_{A\subseteq [n]} m_v(A) \, \int_{D_g} g(\mathbf{u})\,\prod_{i\in A} R_i(t,\mathbf{u})\, d\mathbf{u},\label{eq:RsMobFCInd}\\
\mathrm{MTTF}_S &=& \sum_{A\subseteq [n]} m_v(A) \, \int_{D_g} g(\mathbf{u})\,\int_0^{\infty} \prod_{i\in A} R_i(t,\mathbf{u})\, dt\,
d\mathbf{u}.\label{eq:MTTFsMobFCInd}
\end{eqnarray}

In some cases it is natural to think of the ``intrinsic randomness'' of a unit's lifetime (expressed in the shape of its distribution) as
specific to the unit itself, while the overall parameters of the lifetime (mean, variance, etc.) can be influenced by the external factors. One
way to model that is to consider only the distribution parameters as functions of those factors (without changing the form of the distribution).

For example, the assumption of constant unit failure rate is widespread in the literature and often justified by data. At the same time,
external factors, physical or societal, often influence values of the rates. In the framework of Bayes-dependence, one can model that situation
by setting $R_i(t,\mathbf{u}) = e^{-\lambda_i(\mathbf{u})t}$.

Computing the integrals in the formulas above, one obtains:
\begin{eqnarray*}
R_S(t) &=& \sum_{A\subseteq [n]} m_v(A) \, \int_{D_g} e^{-\lambda_A(\mathbf{u})t}\, g(\mathbf{u})\, d\mathbf{u},\\
\mathrm{MTTF}_S &=& \sum_{A\subseteq [n]} m_v(A) \, \int_{D_g}\frac{g(\mathbf{u})}{\lambda_A(\mathbf{u})}\,
d\mathbf{u},
\end{eqnarray*}
where $\lambda_A(\mathbf{u}) = \sum_{i\in A}\lambda_i(\mathbf{u})$.

\begin{remark}
The above results actually require merely that the joint distribution function of the components' lifetimes should have the integrated product
form of (\ref{eq:CdfQuasiProd}) only when the arguments of the distribution function are either equal to $t$ or $\infty$.
\end{remark}

\section{Systems with pre-phase}

In practice and literature many cases emerge where the units' failure rates should be modeled as nonconstant, say, increasing with elapsed time.
A popular way to address that issue is to present the rates as functions of the elapsed time using, for example, Weibull distribution. However,
so constructed models are, for obvious reasons, much harder to analyze mathematically and provide reliable estimates of parameters
statistically. Moreover, in the context of reliability systems it is impossible to ignore the influence of joint functioning and interaction on
individual units' failure rates.

The framework of Bayes-dependence yields a natural way to make a model where both the joint functioning and elapsed time are accounted for and
the analytical convenience of exponential distribution can be retained.

Let each component's lifetime $T_i$ consist of a random variable (``pre-phase'') $U$ common for all components followed by, once the pre-phase
is over at some time $u$, the subsequent individual ``decay'' phase $Y_i$. Thus, both the variability of the system's failure rate and the
in-system interaction effect can be modeled through the pre-phase distribution, leaving the analyst enough room to model the residual lifetime
of the system separately.

We will denote by $G(u)$, $g(u)$, and $E[U]$ the distribution and density functions and expectation of $U$. The conditional distribution
function of each decay phase given that $U = u$ will be denoted by $F_i (y,u)$.

This case, too, belongs in the class of Bayes-dependence, where the pre-phase duration $U$ plays the role of an external factor and
\begin{equation}\label{eq:s7f6sdfdf}
R_i(t,u) = 1-F_i(|t-u|_+,u),
\end{equation}
where we denote $|x|_+ = \max(x,0)$.

%
%

Substituting (\ref{eq:s7f6sdfdf}) in (\ref{eq:RsMobFCInd}) and using the identity $\sum_{A\subseteq [n]}m_v(A)=1$, we obtain
\begin{eqnarray*}
R_S(t) &=& \sum_{A\subseteq [n]}m_v(A)\int_0^{\infty}g(u)\,\prod_{i\in A}R_i(t,u)\, du\\
&=& \sum_{A\subseteq [n]}m_v(A)\int_t^{\infty}g(u)\, du + \sum_{A\subseteq [n]}m_v(A)\int_0^t g(u)\,\prod_{i\in A}\big(1-F_i(t-u,u)\big)\, du\\
&=& 1-G(t) + \int_0^t\sum_{A\subseteq [n]}m_v(A)\,\prod_{i\in A}\big(1-F_i(t-u,u)\big)g(u)\, du.
\end{eqnarray*}
Thus, we have
$$
R_S(t)=1-G(t)+\int_0^t R^*_S(t-u,u)\, g(u)\, du,
$$
where
$$
R^*_S(y,u)=\sum_{A\subseteq [n]}m_v(A)\,\prod_{i\in A}\big(1-F_i(y,u)\big).
$$

From this result we obtain
\begin{eqnarray*}
\mathrm{MTTF}_S &=& \int_0^{\infty}R_S(t)\, dt\\
&=& \int_0^{\infty}\big(1-G(t)\big)\, dt + \int_0^{\infty}\int_0^t R^*_S(t-u,u)\, g(u)\, du\, dt\\
&=& E[U] + \int_0^{\infty}\int_u^{\infty}R^*_S(t-u,u)\, dt\, g(u)\, du\\
&=& E[U] + \int_0^{\infty}\int_0^{\infty}R^*_S(t,u)\, dt\, g(u)\, du
\end{eqnarray*}
and hence
$$
\mathrm{MTTF}_S = E[U] + \int_0^{\infty}\mathrm{MTTF}^*_S(u)\, g(u)\, du,
$$
where
$$
\mathrm{MTTF}^*_S(u) = \int_0^{\infty}R^*_S(t,u)\, dt = \sum_{A\subseteq [n]}m_v(A)\,\int_0^{\infty}\prod_{i\in A}\big(1-F_i(t,u)\big)\, dt.
$$

When all individual decay phases have exponential distribution with parameters $\lambda_i(u)$, that is $F_i(y,u)=1-e^{-\lambda_i(u)y}$, the
formulas above reduce to
\begin{eqnarray*}
R_S(t) &=& 1-G(t)+\sum_{A\subseteq [n]} m_v(A) \, \int_0 ^t e^{-\lambda_A (u)|t-u|_+}\, g(u)\, du,\\
\mathrm{MTTF}_S &=& E[U] + \sum_{A\subseteq [n]} m_v(A)\, \int_0 ^{\infty}\frac{g(u)}{\lambda_A(u)}\, du,
\end{eqnarray*}
where $\lambda_A (u) = \sum_{i\in A}\lambda_i(u)$.

The previous formulas provide analytically simple results when the pre-phase' distribution is modeled using a piecewise constant density
function. For example, in the simplest case where the pre-phase is uniformly distributed in an interval $[a,b]$ and decay failure rates are
constant we have
$$
R_S(t)=%
\begin{cases}
1, & \mbox{if $t<a$},\cr \displaystyle{\frac{b-t}{b-a}+\sum_{A\subseteq[n]} m_v(A)\,\frac{1 - e^{\lambda_A(a-t)}}{\lambda_A (b-a)}}, & \mbox{if
$a\leqslant t\leqslant b$},\cr \displaystyle{\sum_{A\subseteq[n]} m_v(A)\, e^{-\lambda_A t}\,\frac{e^{\lambda_A b} - e^{\lambda_A a}}{\lambda_A(b-a)}}, & \mbox{if
$b<t$}.\cr
\end{cases}
$$

\section{Systems with collective bounds}

The environment in which a system is installed typically imposes its constraints on the functioning of the system. A power source that feeds
several components imposes an upper bound on their service duration -- it becomes a maximum of the component's own lifetime and the source's
one.

There is extensive reliability literature on a model where upper bounds are imposed by fatally disabling external shocks to individual components and subsystems. The shock emergence processes are assumed Poisson, independent of the system and independent for different components and subsystems. It has been shown that in the reliability terms the action of shocks on the system is equivalent to assuming a Marshall-Olkin type distribution for the component lifetimes. (One can see, e.g., \cite{MulSca87}.)

In the context of insurance applications, a contract can be purchased with an on call emergency service that backs up services of some
components. It provides a lower bound on their service duration, that is, component's service duration becomes a minimum of its intrinsic
lifetime and the emergency service's one.

Today, in a complex system component services may have a complex combination of upper and lower bounds imposed on them by the nature of the
system's purpose and its environment. In general, the presence of collective bounds and interaction of components with them, all the maximums
and minimums that emerge that way, can be represented as follows.

Denote by $Q_1, Q_2,\ldots,Q_m$ bounding random variables and by $T^0_i$  -- the intrinsic lifetime of the $i$th component (if left alone), the
component's service duration emerging as a result of interaction of its intrinsic lifetime and relevant bounding factors can be expressed using
an l.p.\ function $q_i$
$$
T_i = q_i (T^0 _i, Q_1,\ldots,Q_m).
$$

Using that in Theorem~\ref{thm:SFvsLP} and $p$ as the l.p.\ function of the system, the overall system's time-to-failure can be written as
\begin{equation}\label{eq:TSBounds}
T_S = p\big(q_1(T^0_1, Q_1, \ldots,Q_m),\ldots, q_n (T^0 _n, Q_1,\ldots,Q_m)\big).
\end{equation}

\begin{remark}
When all bounding variables are constant equation (\ref{eq:TSBounds}) presents $T_S$ as a so-called ``weighted'' lattice polynomial function of
intrinsic lifetimes; see \cite{DukMar08}.
\end{remark}

A composition of l.p.\ functions, it follows immediately from the definition, is an l.p.\ function itself. We will denote the one in
(\ref{eq:TSBounds}) by $p^a$. It corresponds to what we will refer to as ``the augmented system'', whose units are of two kinds: the original
system's components forming a set $[n]$ and the set of binding factors denoted by $[m]$. Now equation (\ref{eq:TSBounds}) becomes
$$
T_S = p^a(T^0_1,\ldots,T^0_n,Q_1,\ldots,Q_m).
$$

Under a natural assumption that intrinsic components' lifetimes are independent of the bounding factors, the augmented system falls into the
class of Bayes-dependent systems where the bounding variables play the role of dependence-inducing factors. To apply the relevant formulas of
Sections 2 and 3 to this case the following notation will be needed.

\begin{itemize}
\item Joint reliability function of bounding variables: $$R_b(\mathbf{t}) = \Pr(Q_j>t_j\;\forall j\in [m]);$$

\item Component subsets of the augmented system: $A \oplus B, A\subseteq [n], B\subseteq [m]$;

\item $v^a$: the ``$v$''-function corresponding to $p^a$.
\end{itemize}

Now formulas (\ref{eq:RsMobFCInd}) and (\ref{eq:MTTFsMobFCInd}) yield that

\begin{eqnarray}
R_S(t) &=& \sum_{A\subseteq [n]} \sum_{B\subseteq [m]} m_{v^a}(A \oplus B)\, R_b(\mathbf{e}^{0,t}_B)\,\prod_{i\in A} R_i(t),\label{eq:RsMobBounds}\\
\mathrm{MTTF}_S &=& \sum_{A\subseteq [n]}\,\sum_{B\subseteq [m]} m_{v^a}(A\oplus B)\,\int_0^{\infty} R_b(\mathbf{e}^{0,t}_B) \,\prod_{i\in A}
R_i(t)\, dt,\label{eq:MTTFsMobBounds}
\end{eqnarray}

A practically important special case arises under additional assumptions that the original components' intrinsic lifetimes have exponential
distributions and that bounding factors are independent of each other. Then
$$
R_b(\mathbf{e}^{0,t}_B) = \prod_{j\in B} R_j(t)
$$
and formulas (\ref{eq:RsMobBounds}) and (\ref{eq:MTTFsMobBounds}) yield the following:
\begin{eqnarray*}
R_S(t) &=& \sum_{A\subseteq [n]}\,\sum_{B\subseteq [m]} m_v(A\oplus B)\, e^{-\lambda_A t}\prod_{j\in B} R_j(t),\\
\mathrm{MTTF}_S &=& \sum_{A\subseteq [n]}\,\sum_{B\subseteq [m]} m_v(A\oplus B)\,\int_0^{\infty} e^{-\lambda_A t}\, \prod_{j\in B} R_j(t)\, dt,
\end{eqnarray*}
where, as before, $\lambda_A = \sum_{i\in A}\lambda_i$.

\section*{Acknowledgments}

The authors thank the referees for constructive suggestions. Jean-Luc Marichal acknowledges support by the internal research project
F1R-MTH-PUL-09MRDO of the University of Luxembourg.


\end{document}